\documentclass[11pt,reqno]{article}
\usepackage{amsmath,amssymb,amsthm}
\newtheorem{theorem}{Theorem}[section]
\newtheorem{lemma}[theorem]{Lemma}

\newtheorem{corollary}[theorem]{Corollary}

\newtheorem{remark}[theorem]{Remark}
\usepackage[english]{babel}
\numberwithin{equation}{section}

\title{\bf Small data global well-posedness for the inhomogeneous biharmonic NLS in Sobolev spaces}
\author{{JinMyong An, PyongJo Ryu, JinMyong Kim$^*$}\\
\footnotesize{Faculty of Mathematics, {\bf Kim Il Sung} University, Pyongyang, Democratic People's Republic of Korea}\\
\footnotesize{$^*$ Corresponding Author: JinMyong Kim (jm.kim0211@ryongnamsan.edu.kp)}}
\date{}
\begin{document}
\maketitle
\begin{abstract}
In this paper, we study the Cauchy problem for the inhomogeneous biharmonic nonlinear Schr\"{o}dinger equation (IBNLS)
\[iu_{t} +\Delta^{2} u=\lambda |x|^{-b}|u|^{\sigma}u,u(0)=u_{0} \in H^{s} (\mathbb R^{d}),\]
where  $\lambda \in \mathbb R$, $d\in \mathbb N$, $0<s<\min \{2+\frac{d}{2},\frac{3}{2}d\}$ and $0<b<\min\{4,d,\frac{3}{2}d-s,\frac{d}{2}+2-s\}$.
Under some regularity assumption for the nonlinear term, we prove that the IBNLS equation is globally well-posed in $H^{s}(\mathbb R^{d})$ if $\frac{8-2b}{d}<\sigma< \sigma_{c}(s)$ and the initial data is sufficiently small, where $\sigma_{c}(s)=\frac{8-2b}{d-2s}$ if $s<\frac{d}{2}$, and $\sigma_{c}(s)=\infty$ if $s\ge \frac{d}{2}$.
\end{abstract}


\section{Introduction}\label{sec 1.}

In this work, we study the Cauchy problem for the inhomogeneous biharmonic nonlinear Schr\"{o}dinger equation (IBNLS)
\begin{equation} \label{GrindEQ__1_1_}
\left\{\begin{array}{l} {iu_{t} +\Delta^{2}u=\lambda |x|^{-b} |u|^{\sigma} u,~(t,x)\in \mathbb R\times \mathbb R^{d},}\\
{u(0,x)=u_{0}(x)\in H^{s}(\mathbb R^{d})}, \end{array}\right.
\end{equation}
where $d\in \mathbb N$, $s\ge 0$, $0<b<4$, $\sigma>0$ and $\lambda \in \mathbb R$.
The IBNLS equation \eqref{GrindEQ__1_1_} has the following equivalent form:
\begin{equation}\nonumber
u(t)=e^{it\Delta^{2}}u_{0} -i\lambda \int _{0}^{t}e^{i(t-\tau)\Delta^{2}}|x|^{-b}|u(\tau)|^{\sigma}u(\tau)d\tau .
\end{equation}
The limiting case $b=0$ is the well-known classic biharmonic nonlinear Schr\"{o}dinger equation, also called the fourth-order nonlinear Schr\"{o}dinger equation. It was introduced by Karpman \cite{K96} and Karpman-Shagalov \cite{KS97} to take into account the role of small fourth-order dispersion terms in the propagation of intense laser beams in a bulk medium with Kerr nonlinearity and it has been widely studied during the last two decades. See, for example, \cite{D18B,D21,LZ21,P07} and the references therein.
The equation \eqref{GrindEQ__1_1_} has a counterpart for the Laplacian operator, namely, the inhomogeneous nonlinear Schr\"{o}dinger equation (INLS)
\begin{equation} \label{GrindEQ__1_2_}
iu_{t} +\Delta u=\lambda |x|^{-b} |u|^{\sigma} u.
\end{equation}
The INLS equation \eqref{GrindEQ__1_2_} arises in nonlinear optics for modeling the propagation of laser beam and it has attracted a lot of interest in recent years. See, for example, \cite{AT21,AK211,AK212,AKC21,C21,DK21,GM21} and the references therein.

First of all, we recall some facts about this equation. The IBNLS equation \eqref{GrindEQ__1_1_} enjoys the conservations of mass and energy, which are defined respectively by
\begin{equation} \label{GrindEQ__1_3_}
M\left(u(t)\right):=\int_{\mathbb R^{d}}{| u(t,x)|^{2}dx}=M\left(u_{0} \right),
\end{equation}
\begin{equation} \label{GrindEQ__1_4_}
E\left(u(t)\right):=\frac{1}{2}\int_{\mathbb R^{d}}{|\Delta u(t,x)|^2 dx}-\frac{\lambda
}{\sigma+2}\int_{\mathbb R^{d}}{|x|^{-b}\left|u(t,x)\right|^{\sigma+2}dx}=E\left(u_{0} \right).
\end{equation}
The IBNLS equation \eqref{GrindEQ__1_1_} is invariant under scaling $u_{\alpha}(t,x)=\alpha^{\frac{4-b}{\sigma}}u(\alpha^{4}t,\alpha x ),~\alpha >0$.
An easy computation shows that
$$
\left\|u_{\alpha}(t)\right\|_{\dot{H}^{s}}=\alpha^{s+\frac{4-b}{\sigma}-\frac{d}{2}}\left\|u(t)\right\|_{\dot{H}^{s}}.
$$
Hence, the critical Sobolev index is defined as follows:
\begin{equation} \label{GrindEQ__1_5_}
s_{c}:=\frac{d}{2}-\frac{4-b}{\sigma}.
\end{equation}
Putting
\begin{equation} \label{GrindEQ__1_6_}
\sigma_{c}(s):=
\left\{\begin{array}{cl}
{\frac{8-2b}{d-2s},} ~&{{\rm if}~s<\frac{d}{2},}\\
{\infty,}~&{{\rm if}~s\ge \frac{d}{2},}
\end{array}\right.
\end{equation}
we can easily see that $s>s_{c}$ is equivalent to $\sigma<\sigma_{c}(s)$. If $s<\frac{d}{2}$, $s=s_{c}$ is equivalent to $\sigma=\sigma_{c}(s)$.
For initial data $u_{0}\in H^{s}(\mathbb R^{d})$, we say that the Cauchy problem \eqref{GrindEQ__1_1_} is $H^{s}$-critical (for short, critical) if $0\le s<\frac{d}{2}$ and $\sigma=\sigma_{c}(s)$.
If $s\ge 0$ and $\sigma<\sigma_{c}(s)$, the problem \eqref{GrindEQ__1_1_} is said to be  $H^{s}$-subcritical (for short, subcritical).
Especially, if $\sigma =\frac{8-2b}{d}$, the problem is known as $L^{2}$-critical or mass-critical.
If $\sigma =\frac{8-2b}{d-4}$ with $d\ge 5$, it is called $H^{2}$-critical or energy-critical.
Throughout the paper, a pair $(p,q)$ is said to be  admissible, for short $(p,q)\in A$, if
\begin{equation}\label{GrindEQ__1_7_}
\frac{2}{p}+\frac{d}{q}\le\frac{d}{2},~p\in [2,\infty],~q\in [2,\infty).
\end{equation}
Especially, we say that $(p,q)\in A$ is biharmonic admissible (for short $(p,q)\in B$), if $\frac{4}{p}+\frac{d}{q}=\frac{d}{2}$. If $(p,q)\in A$ and $\frac{2}{p}+\frac{d}{q}=\frac{d}{2}$, then $(p,q)$ is said to be Schr\"{o}dinger admissible (for short $(p,q)\in S$).
We also denote for $(p,q)\in [1,\infty]^{2}$,
\begin{equation}\label{GrindEQ__1_8_}
\gamma_{p,q}=\frac{d}{2}-\frac{4}{p}-\frac{d}{q}.
\end{equation}

The IBNLS equation \eqref{GrindEQ__1_1_} has attracted a lot of interest in recent years. See, for example, \cite{CG21, CGP20, GP20, GP21, LZ212, S21} and the references therein.
Guzm\'{a}n-Pastor \cite{GP20} proved that \eqref{GrindEQ__1_1_} is locally well-posed in $L^{2}$, if $d\in \mathbb N$, $0<b<\min\left\{4,d\right\}$ and $0<\sigma<\sigma_{c}(0)$. Using the conservation of mass, this result is directly applied to obtain the global well-posedness result in $L^{2}$. They also established the local well-posedness in $H^{2}$ for $d\ge 3$, $0<b<\min\{4,\frac{d}{2}\}$, $\max\{0,\frac{2-2b}{d}\}<\sigma<\sigma_{c}(2)$.
Using the above local well-posedness result and the standard argument, they established the global well-posedness in $H^{2}$ for $\max\left\{0,\frac{2-2b}{d}\right\}<\sigma\le \frac{8-2b}{d}$. They also established the small data global well-posedness in $H^{2}$ for \eqref{GrindEQ__1_1_} in the following cases:

$\cdot$ $d\ge 8$, $0<b<4$ and $\frac{8-2b}{d}<\sigma<\frac{8-2b}{d-4}$,

$\cdot$ $d=5,6,7$, $0<b<\frac{d^2-8d+32}{8}$ and $\frac{8-2b}{d}<\sigma<\frac{d-2b}{d-4}$,

$\cdot$ $d=6,7$, $0<b<d-4$ and $\frac{8-2b}{d}<\sigma<\frac{8-2b}{d-4}$,

$\cdot$ $d=3,4$, $0<b<\frac{d}{2}$ and $\frac{8-2b}{d}<\sigma<\infty$.

\noindent The global well-posedness and scattering in $H^{2}$ in the intercritical case $\frac{8-2b}{d}<\sigma<\sigma_{c}(2)$ with $d\ge 3$ were also studied in \cite{CG21,GP21, S21}.
Afterwards, Cardoso-Guzm\'{a}n-Pastor \cite{CGP20} established the local and global well-posedness of \eqref{GrindEQ__1_1_} in $\dot{H}^{s_{c}}\cap \dot{H}^{2}$ with $d\ge 5$, $0<s_{c}<2$, $0<b<\min\left\{4,\frac{d}{2}\right\}$ and $\max\left\{1,\frac{8-2b}{d}\right\}<\sigma< \frac{8-2b}{d-4}$.
Liu-Zhang \cite{LZ212} established the local well-posedness in $H^{s}$ with $0<s\le 2$ by using the Besov space theory.
More precisely, they proved that the IBNLS equation \eqref{GrindEQ__1_1_} is locally well-posed in $H^{s}$ if $d\in \mathbb N$, $0<s\le 2$, $0<\sigma<\sigma_{c}(s)$, $0<b<\min\left\{4,\frac{d}{2}\right\}$. See Theorem 1.5 of \cite{LZ212} for details.
This result about the local well-posedness of \eqref{GrindEQ__1_1_} improves the one of \cite{GP20} by not only extending the validity of $d$ and $s$ but also removing the lower bound $\sigma>\frac{2-2b}{d}$.
This local well-posedness result is directly applied to obtain the global well-posedness result in $H^{2}$ in the full range of mass-subcritical and mass-critical cases $0<\sigma\le \frac{8-2b}{d}$ (see Corollary 1.10 of \cite{LZ212}).
Recently, the authors in \cite{ARK22} proved that the IBNLS equation \eqref{GrindEQ__1_1_} is locally well-posed in $H^s(\mathbb R^d)$, if
$d\in \mathbb N$, $0\le s <\min \left\{2+\frac{d}{2},\frac{3}{2}d\right\}$, $0<b<\min\left\{4,d,\frac{3}{2}d-s,\frac{d}{2}+2-s\right\}$, $0<\sigma<\sigma_{c}(s)$ and the following regularity assumption for the nonlinear term is further assumed:
\begin{equation} \label{GrindEQ__1_9_}
\sigma~\textrm{is an even integer, or}~ \sigma\ge \sigma_{\star}(s),
\end{equation}
where
\footnote[1]{Given $a\in \mathbb R$, $a^{+}$ denotes the fixed number slightly larger than $a$ ($a^{+}=a+\varepsilon$ with $\varepsilon>0$ small enough).}
\begin{equation} \label{GrindEQ__1_10_}
\sigma_{\star}(s):=
\left\{\begin{array}{ll}
{0,}~&{{\rm if}~{d\in \mathbb N ~{\rm and}~s\le 1,}}\\
{\left[ s-\frac{d}{2}\right]},~&{{\rm if}~d=1,2 ~{\rm and}~ s\ge 1+\frac{d}{2}},\\
{\lceil s\rceil-2},~&{{\rm if}~d\ge 3~{\rm and}~s>2}\\
{\left(\frac{2s-2b-2}{d}\right)^{+}},~&{\rm otherwise.}
\end{array}\right.
\end{equation}
As the continuation of the work \cite{ARK22}, this paper investigates the small data global well-posedness and scattering in $H^s$ with $0< s <\min \left\{2+\frac{d}{2},\frac{3}{2}d\right\}$ for the IBNLS equation \eqref{GrindEQ__1_1_}. To arrive at this goal, we establish the delicate nonlinear estimates (see Lemmas \ref{lem 3.2.} and \ref{lem 3.3.}) and use the contraction mapping principle based on Strichartz estimates with a gain of derivatives (see Lemma \ref{lem 2.7.}). Our main result is the following.
\begin{theorem}\label{thm 1.1.}
Let $d\in \mathbb N$, $0< s <\min \left\{2+\frac{d}{2},\frac{3}{2}d\right\}$, $0<b<\min\left\{4,d,\frac{3}{2}d-s,\frac{d}{2}+2-s\right\}$ and $\frac{8-2b}{d}<\sigma<\sigma_{c}(s)$. If $\sigma$ is not an even integer, assume further that
\begin{equation} \label{GrindEQ__1_11_}
\sigma\ge \tilde{\sigma}_{\star}(s):=
\left\{\begin{array}{ll}
{0,}~&{{\rm if}~d\in \mathbb N ~{\rm and}~s\le 1,}\\
{\left[ s-\frac{d}{2}\right],}~&{{\rm if}~d=1,2 ~{\rm and}~ s>1},\\
{\lceil s\rceil-2},~&{{\rm if}~d\ge 3~{\rm and}~s>1.}
\end{array}\right.
\end{equation}
Then for any $u_{0} \in H^{s}$ satisfying $\left\|u_{0}\right\|_{H^s}<\delta$ for some $\delta>0$ small enough, there exists a unique, global solution of \eqref{GrindEQ__1_1_} satisfying
\begin{equation} \label{GrindEQ__1_12_}
u\in C\left(\mathbb R,H^{s} \right)\cap L^{p} \left(\mathbb R,H_{q}^{s} \right),
\end{equation}
for any $(p,q)\in B$.
Furthermore, there exist $u_{0}^{\pm } \in H^{s}$ such that
\begin{equation} \label{GrindEQ__1_13_}
{\mathop{\lim }\limits_{t\to \pm \infty }} \left\| u(t)-e^{it\Delta^2 } u_{0}^{\pm } \right\| _{H^{s}} =0.
\end{equation}
\end{theorem}
As an immediate consequence of Theorem \ref{thm 1.1.}, we have the following small data global well-posedness and scattering in $H^{2}$.
\begin{corollary}\label{cor 1.2.}
Let $d\ge 2$, $0<b<\min\left\{4,\frac{d}{2}\right\}$ and $\frac{8-2b}{d}<\sigma<\sigma_{c}(2)$. Then for any $u_{0} \in H^{2}$ satisfying $\left\|u_{0}\right\|_{H^2}<\delta$ for some $\delta>0$ small enough, there exists a unique, global solution of \eqref{GrindEQ__1_1_} satisfying
\begin{equation} \nonumber
u\in C\left(\mathbb R,H^{2} \right)\cap L^{p} \left(\mathbb R,H_{q}^{2} \right),
\end{equation}
for any $(p,q)\in B$.
Moreover, there exist $u_{0}^{\pm } \in H^{2}$ such that
\begin{equation} \nonumber
{\mathop{\lim }\limits_{t\to \pm \infty }} \left\| u(t)-e^{it\Delta^2 } u_{0}^{\pm } \right\| _{H^{2}} =0.
\end{equation}
\end{corollary}
\begin{remark}\label{rem 1.3.}
\textnormal{The global well-posedness in $H^2$ for \eqref{GrindEQ__1_1_} with $\sigma\le\frac{8-2b}{d}$ were obtained by \cite{LZ212}. See Corollary 1.10 of \cite{LZ212} for details.}
\end{remark}
\begin{remark}\label{rem 1.4.}
\textnormal{Corollary \ref{cor 1.2.} improves the small data global well-posedness results of \cite{CG21, GP20, GP21} by extending the validity of $d$, $b$ and $\sigma$.}
\end{remark}

This paper is organized as follows. In Section \ref{sec 2.}, we recall some notation and give some preliminary results related to our problem. In Section \ref{sec 3.}, we establish the various nonlinear estimates to prove Theorem \ref{thm 1.1.}.

\section{Preliminaries}\label{sec 2.}
We begin by introducing some basic notation.
$C>0$ expresses the universal positive constant, which can be different at different places.
The notation $a\lesssim b$ means $a\le Cb$ for some constant $C>0$.
For $p\in \left[1,\infty \right]$, $p'$ denotes the dual number of $p$, i.e. $1/p+1/p'=1$.
For $s\in\mathbb R$, we denote by $[s]$ the largest integer which is less than or equal to $s$ and by $\left\lceil s\right\rceil $ the minimal integer which is larger than or equal to $s$.
As in \cite{WHHG11}, for $s\in \mathbb R$ and $1<p<\infty $, we denote by $H_{p}^{s} (\mathbb R^{d} )$ and $\dot{H}_{p}^{s} (\mathbb R^{d} )$ the nonhomogeneous Sobolev space and homogeneous Sobolev space, respectively. As usual, we abbreviate $H_{2}^{s} (\mathbb R^{d} )$ and $\dot{H}_{2}^{s} (\mathbb R^{d} )$ as $H^{s} (\mathbb R^{d} )$ and $\dot{H}^{s} (\mathbb R^{d} )$, respectively. Given two normed spaces $X$ and $Y$, $X\hookrightarrow Y$ means that $X$ is continuously embedded in $Y$, i.e. there exists a constant $C\left(>0\right)$ such that $\left\| f\right\| _{Y} \le C\left\| f\right\| _{X} $ for all $f\in X$. If there is no confusion, $\mathbb R^{d} $ will be omitted in various function spaces.

Next, we recall some useful facts which are used throughout the paper.
First of all, we recall some useful embeddings on Sobolev spaces. See \cite{WHHG11} for example.

\begin{lemma}\label{lem 2.1.}
Let $-\infty <s_{2} \le s_{1} <\infty $ and $1<p_{1} \le p_{2} <\infty $ with $s_{1} -\frac{d}{p_{1} } =s_{2} -\frac{d}{p_{2} } $. Then we have the following embeddings:
\[\dot{H}_{p_{1} }^{s_{1} } \hookrightarrow \dot{H}_{p_{2} }^{s_{2} } ,~H_{p_{1} }^{s_{1} } \hookrightarrow H_{p_{2} }^{s_{2} } .\]
\end{lemma}

\begin{lemma}\label{lem 2.2.}
Let $-\infty <s<\infty $ and $1<p<\infty $. Then we have
\begin{enumerate}
\item $H_{p}^{s+\varepsilon } \hookrightarrow H_{p}^{s} $ $(\varepsilon >0)$,
\item $H_{p}^{s} =L^{p} \cap \dot{H}_{p}^{s}$ $(s>0)$.
\end{enumerate}
\end{lemma}

\begin{corollary}\label{cor 2.3.}
Let $-\infty <s_{2} \le s_{1} <\infty $ and $1<p_{1} \le p_{2} <\infty $ with $s_{1} -\frac{d}{p_{1} } \ge s_{2} -\frac{d}{p_{2} } $. Then we have $H_{p_{1} }^{s_{1} } \hookrightarrow H_{p_{2} }^{s_{2} } $.
\end{corollary}
\begin{proof} The result follows from Lemma \ref{lem 2.1.} and Item 1 of Lemma \ref{lem 2.2.}.
\end{proof}

\begin{lemma}[\cite{K95}]\label{lem 2.4.}
Let $F\in C^{k}(\mathbb C,\mathbb C)$, $k\in \mathbb N \setminus \{0\}$. Assume that there is $\nu\ge k$ such that
$$
|D^{i}F(z)|\lesssim |z|^{\nu-i},~z\in \mathbb C,~i=1,2,\ldots,k.
$$
Then for $s \in [0,k]$ and $1<r,p<\infty$, $1<q\le \infty$ satisfying $\frac{1}{r}=\frac{1}{p}+\frac{\nu-1}{q}$, we have
$$
\left\|F(u)\right\|_{\dot{H}_{r}^{s}}\lesssim \left\|u\right\|_{L^q}^{\nu-1}\left\|u\right\|_{\dot{H}_{p}^{s}}.
$$
\end{lemma}

Lemma \ref{lem 2.4.} applies in particular to the model case $F(u)=\lambda|u|^{\sigma}u$ with $\lambda\in \mathbb R$.
\begin{corollary}\label{cor 2.5.}
Let $\sigma>0$, $s\ge 0$ and $1<r,p<\infty$, $1<q\le \infty$ satisfying $\frac{1}{r}=\frac{1}{p}+\frac{\sigma}{q}$. If $\sigma$ is not an even integer, assume that $\sigma\ge \lceil s\rceil-1$. Then we have
$$
\left\||u|^{\sigma}u\right\|_{\dot{H}_{r}^{s}}\lesssim \left\|u\right\|_{L^q}^{\sigma}\left\|u\right\|_{\dot{H}_{p}^{s}}.
$$
\end{corollary}

Next, we recall the well-known fractional product rule.

\begin{lemma}[Fractional Product Rule, \cite{CW91}]\label{lem 2.6.}
Let $s\ge 0$, $1<r,r_{2},p_{1}<\infty$ and $1<r_{1},p_{2}\le\infty$. Assume that
\[\frac{1}{r} =\frac{1}{r_{i} } +\frac{1}{p_{i} }\;(i=1,2).\]
Then we have
\begin{equation}
\left\| fg\right\| _{\dot{H}_{r}^{s} } \lesssim \left\| f\right\| _{r_{1} } \left\| g\right\| _{\dot{H}_{p_{1} }^{s} } +\left\| f\right\| _{\dot{H}_{r_{2} }^{s} } \left\| g\right\| _{p_{2} } .
\end{equation}
\end{lemma}

We end this section with recalling the Strichartz estimates for the fourth-order Schr\"{o}dinger equation. See, for example, \cite{D18I} and the reference therein.
\begin{lemma}[Strichartz estimates]\label{lem 2.7.}
Let $\gamma\in \mathbb R$ and $u$ be the solution to the linear fourth-order Schr\"{o}dinger equation, namely
\begin{equation}\nonumber
u(t)=e^{it\Delta^{2}}u_{0}+i\int_{0}^{t}{e^{i(t-s)\Delta^{2}}F(s)ds,}
\end{equation}
for some data $u_{0}$ and $F$. Then for all $(p,q)$ and $(a,b)$ admissible,
$$
\left\|u\right\|_{L^{p}(\mathbb R,\dot{H}_{q}^{\gamma})}\lesssim  \left\|u_{0}\right\|_{\dot{H}^{\gamma+\gamma_{p,q}}}+ \left\|F\right\|_{L^{a'}(\mathbb R,\dot{H}_{b'}^{\gamma+\gamma_{p,q}-\gamma_{a',b'}-4})},
$$
where $\gamma_{p,q}$ and $\gamma_{a',b'}$ are as in \eqref{GrindEQ__1_8_}.
\end{lemma}
\section{Proof of main result}\label{sec 3.}

In this section, we prove Theorem \ref{thm 1.1.}. Before establishing the various nonlinear estimates, we recall the following useful fact. See Remark 3.4 of \cite{AK212}.
\begin{remark}\label{rem 3.1.}
\textnormal{Let $1<r<\infty$, $s\ge0$ and $b>0$. Let $\chi\in C_{0}^{\infty}(\mathbb R^{d})$ satisfy $\chi(x)=1$ for $|x|\le 1$ and $\chi(x)=0$ for $|x|\ge 2$.
If $b+s<\frac{d}{r}$, then $\chi(x)|x|^{-b} \in \dot{H}_{r}^{s}$.
If $b+s>\frac{d}{r}$, then $(1-\chi(x))|x|^{-b}\in \dot{H}_{r}^{s}$.}
\end{remark}
\begin{lemma}\label{lem 3.2.}
Let $d\in \mathbb N$, $s>0$, $0<b<\min\left\{4,d\right\}$ and $\frac{8-2b}{d}<\sigma<\sigma_{c}(s)$.
Then there exist $(a_{i},b_{i})\in B$, $(p_{i},q_{i})\in B$ and $(\bar{p}_{i},\bar{q}_{i})\in B$ $(i=1,2)$ such that
\begin{equation} \label{GrindEQ__3_1_}
\left\|\chi(x)|x|^{-b}|u|^{\sigma}v\right\|_{L^{a_{1}'}(\mathbb R, L^{b_{1}'})}
\lesssim \left\|u\right\|^{\sigma}_{L^{p_{1}}(\mathbb R, H_{q_{1}}^{s})}\left\|v\right\|_{L^{\bar{p}_{1}}(\mathbb R, L^{\bar{q}_{1}})},
\end{equation}
\begin{equation} \label{GrindEQ__3_2_}
\left\|(1-\chi(x))|x|^{-b}|u|^{\sigma}v\right\|_{L^{a_{2}'}(\mathbb R, L^{b_{2}'})}
\lesssim \left\|u\right\|^{\sigma}_{L^{p_{2}}(\mathbb R, H_{q_{2}}^{s})}\left\|v\right\|_{L^{\bar{p}_{2}}(\mathbb R, L^{\bar{q}_{2}})}.
\end{equation}
where $\chi\in C_{0}^{\infty}(\mathbb R^{d})$ is given in Remark \ref{rem 3.1.}.
\end{lemma}

\begin{proof}
For $i=1,2$, let $\alpha_{i}$, $(a_{i},b_{i})\in B$, $(p_{i},q_{i})\in B$ and $(\bar{p}_{i},\bar{q}_{i})\in B$ satisfy the following system:
\begin{equation} \label{GrindEQ__3_3_}
\left\{\begin{array}{ll}
{\max\left\{0,~\frac{1}{q_{i}}-\frac{s}{d}\right\}<\frac{1}{\alpha_{i}}<\frac{1}{q_{i}},}\\
{\frac{1}{a_{i}'}=\frac{\sigma}{p_{i}}+\frac{1}{\bar{p}_{i}}.}\\
\end{array}\right.
\end{equation}
Then, for $\alpha_{1}$, $(a_{1},b_{1})\in B$, $(p_{1},q_{1})\in B$ and $(\bar{p}_{1},\bar{q}_{1})\in B$ satisfying \eqref{GrindEQ__3_3_} and
\begin{equation} \label{GrindEQ__3_4_}
\frac{1}{b_{1}'}-\frac{\sigma}{\alpha_{1}}-\frac{1}{\bar{q}_{1}}>\frac{b}{d},
\end{equation}
we deduce from the H\"{o}lder inequality, Remark \ref{rem 3.1.} and Corollary \ref{cor 2.3.} that
\begin{eqnarray}\begin{split} \label{GrindEQ__3_5_}
\left\|\chi(x)|x|^{-b}|u|^{\sigma}v\right\|_{L^{b_{1}'}}&\lesssim
\left\|\chi(x)|x|^{-b}\right\| _{L^{\gamma_{1}}} \left\|u\right\|^{\sigma} _{L^{\alpha_{1}}}\left\|v\right\| _{L^{\bar{q}_{1}}}
\lesssim \left\|u\right\|^{\sigma}_{H_{q_{1}}^{s}}\left\|v\right\|_{L^{\bar{q}_{1}}},
\end{split}\end{eqnarray}
where $\frac{1}{\gamma_{1}}:=\frac{1}{b_{1}'}-\frac{\sigma}{\alpha_{1}}-\frac{1}{\bar{q_{1}}}$.
Using the second equation in \eqref{GrindEQ__3_3_}, \eqref{GrindEQ__3_5_} and the H\"{o}lder inequality, we immediately get \eqref{GrindEQ__3_1_}.
Using the similar argument, for $\alpha_{2}$, $(a_{2},b_{2})\in B$, $(p_{2},q_{2})\in B$ and $(\bar{p}_{2},\bar{q}_{2})\in B$ satisfying \eqref{GrindEQ__3_3_} and
\begin{equation} \label{GrindEQ__3_6_}
0<\frac{1}{b_{2}'}-\frac{\sigma}{\alpha_{2}}-\frac{1}{\bar{q}_{2}}<\frac{b}{d},
\end{equation}
we also have
\begin{eqnarray}\begin{split} \label{GrindEQ__3_7_}
\left\|(1-\chi(x))|x|^{-b}|u|^{\sigma}v\right\|_{L^{b_{2}'}}&\lesssim
\left\|(1-\chi(x))|x|^{-b}\right\| _{L^{\gamma_{2}}} \left\|u\right\|^{\sigma} _{L^{\alpha_{2}}}\left\|v\right\| _{L^{\bar{q}_{2}}}
\lesssim \left\|u\right\|^{\sigma}_{H_{q_{2}}^{s}}\left\|v\right\|_{L^{\bar{q}_{2}}},
\end{split}\end{eqnarray}
where $\frac{1}{\gamma_{2}}:=\frac{1}{b_{2}'}-\frac{\sigma}{\alpha_{2}}-\frac{1}{\bar{q_{2}}}$.
Using the second equation in \eqref{GrindEQ__3_3_}, \eqref{GrindEQ__3_7_} and the H\"{o}lder inequality, we immediately get \eqref{GrindEQ__3_2_}.
Hence, it suffices to prove that there exist $\alpha_{i}$, $(a_{i},b_{i})\in B$, $(p_{i},q_{i})\in B$ and $(\bar{p}_{i},\bar{q}_{i})\in B$ $(i=1,2)$ satisfying \eqref{GrindEQ__3_3_}, \eqref{GrindEQ__3_4_} and \eqref{GrindEQ__3_6_}.
In fact, the second equation in \eqref{GrindEQ__3_3_} implies that
\begin{equation} \label{GrindEQ__3_8_}
\frac{\sigma}{q_{i}}=\frac{\sigma+2}{2}-\frac{1}{\bar{q}_{i}}-\frac{4}{d}-\frac{1}{b_{i}}.
\end{equation}
We can easily see that $(p_{i},q_{i})\in B$ $(i=1,2)$ provided that
\begin{equation} \label{GrindEQ__3_9_}
1-\frac{4}{d}+\frac{1}{b_{i}}< \frac{1}{\bar{q}_{i}}<\min\left\{\frac{\sigma+2}{2}-\frac{4}{d}-\frac{1}{b_{i}},~1+\frac{2\sigma}{d}-\frac{4}{d}-\frac{1}{b_{i}}\right\}.
\end{equation}
In view of \eqref{GrindEQ__3_8_}, the first equation in \eqref{GrindEQ__3_3_} is equivalent to
\begin{equation} \label{GrindEQ__3_10_}
\max\left\{\frac{\sigma+2}{2}-\frac{1}{\bar{q}_{i}}-\frac{4}{d}-\frac{1}{b_{i}}-\frac{\sigma s}{d},~0\right\}
<\frac{\sigma}{\alpha_{i}}<\frac{\sigma+2}{2}-\frac{1}{\bar{q}_{i}}-\frac{4}{d}-\frac{1}{b_{i}}.
\end{equation}

{\bf Step 1.} In order to prove \eqref{GrindEQ__3_1_}, it suffices to prove that there exist $\alpha_{1}$, $(a_{1},b_{1})\in B$ and $(\bar{p}_{1},\bar{q}_{1})\in B$ satisfying \eqref{GrindEQ__3_4_}, \eqref{GrindEQ__3_9_} and \eqref{GrindEQ__3_10_}.
Since the equation \eqref{GrindEQ__3_4_} is equivalent to
\begin{equation} \nonumber
\frac{\sigma}{\alpha_{1}}<1-\frac{1}{b_{1}}-\frac{1}{\bar{q}_{1}}-\frac{b}{d},
\end{equation}
we can choose $\alpha_{1}$ satisfying \eqref{GrindEQ__3_4_} and \eqref{GrindEQ__3_11_}, provided that $\sigma<\sigma_{c}(s)$ and
\begin{equation} \label{GrindEQ__3_11_}
\frac{1}{\bar{q}_{1}}<\min\left\{1-\frac{1}{b_{1}}-\frac{b}{d},~\frac{\sigma+2}{2}-\frac{4}{d}-\frac{1}{b_{1}}\right\}.
\end{equation}
Hence, it suffices to prove that there exist $(a_{1},b_{1})\in B$ and $(\bar{p}_{1},\bar{q}_{1})\in B$ satisfying \eqref{GrindEQ__3_9_} and \eqref{GrindEQ__3_11_}.
We divide the study in two cases: $d>4$ and $d\le 4$.

If $d>4$, then it suffices to prove that there exist $(a_{1},b_{1})\in B$ and $\bar{q}_{1}$ satisfying
\begin{equation} \nonumber
\max\left\{\frac{d-4}{2d},~1-\frac{4}{d}-\frac{1}{b_{1}}\right\}<\frac{1}{\bar{q}_{1}}<\min\left\{1-\frac{1}{b_{1}}-\frac{b}{d},
~1+\frac{2\sigma}{d}-\frac{4}{d}-\frac{1}{b_{1}},~\frac{1}{2}\right\},
\end{equation}
which is possible provided that $b<4$ and
\begin{equation} \label{GrindEQ__3_12_}
\frac{1}{2}-\frac{4}{d}<\frac{1}{b_{1}}<\min\left\{\frac{1}{2}+\frac{2}{d}-\frac{b}{d},~\frac{1}{2}-\frac{2}{d}+\frac{2\sigma}{d}\right\}.
\end{equation}
And it is obvious that there exists $(a_{1},b_{1})\in B$ satisfying \eqref{GrindEQ__3_12_}.

If $d\le 4$, then it suffices to prove that there exist $(a_{1},b_{1})\in B$ and $\bar{q}_{1}$ satisfying
\begin{equation} \nonumber
0<\frac{1}{\bar{q}_{1}}<\min\left\{1-\frac{1}{b_{1}}-\frac{b}{d},
~1+\frac{\sigma}{2}-\frac{4}{d}-\frac{1}{b_{1}},~\frac{1}{2}\right\},
\end{equation}
which is possible provided that $b<4$ and
\begin{equation} \label{GrindEQ__3_13_}
\frac{1}{b_{1}}<\min\left\{1-\frac{b}{d},~1+\frac{\sigma}{2}-\frac{4}{d}\right\}.
\end{equation}
Using the fact that $b<d$ and $\sigma>\frac{8-2b}{d}$, we can easily see that there exists $(a_{1},b_{1})\in B$ satisfying \eqref{GrindEQ__3_13_}. This completes the proof of \eqref{GrindEQ__3_1_}.

{\bf Step 2.} In order to prove \eqref{GrindEQ__3_2_}, it suffices to prove that there exist $\alpha_{2}$, $(a_{2},b_{2})\in B$ and $(\bar{p}_{2},\bar{q}_{2})\in B$ satisfying \eqref{GrindEQ__3_6_}, \eqref{GrindEQ__3_9_} and \eqref{GrindEQ__3_10_}.
Noticing that the equation \eqref{GrindEQ__3_6_} is equivalent to
\begin{equation} \nonumber
1-\frac{1}{b_{2}}-\frac{1}{\bar{q}_{2}}-\frac{b}{d}<\frac{\sigma}{\alpha_{2}}<1-\frac{1}{b_{2}}-\frac{1}{\bar{q}_{2}},
\end{equation}
we can choose $\alpha_{2}$ satisfying \eqref{GrindEQ__3_6_} and \eqref{GrindEQ__3_10_}, provided that $\frac{8-2b}{d}<\sigma<\sigma_{c}(s)$ and
\begin{equation} \label{GrindEQ__3_14_}
\frac{1}{\bar{q}_{2}}<\min\left\{1-\frac{1}{b_{2}},~\frac{\sigma+2}{2}-\frac{4}{d}-\frac{1}{b_{2}}\right\}.
\end{equation}
Hence, it suffices to prove that there exist $(a_{2},b_{2})\in B$ and $(\bar{p}_{2},\bar{q}_{2})\in B$ satisfying \eqref{GrindEQ__3_9_} and \eqref{GrindEQ__3_14_}.
We divide the study in two cases: $d>4$ and $d\le 4$.

If $d>4$, then it suffices to prove that there exist $(a_{2},b_{2})\in B$ and $\bar{q}_{2}$ satisfying
\begin{equation} \nonumber
\max\left\{\frac{d-4}{2d},~1-\frac{4}{d}-\frac{1}{b_{2}}\right\}<\frac{1}{\bar{q}_{2}}<\min\left\{1-\frac{1}{b_{2}},
~1+\frac{2\sigma}{d}-\frac{4}{d}-\frac{1}{b_{2}},~\frac{1}{2}\right\},
\end{equation}
which is possible provided that
\begin{equation} \label{GrindEQ__3_15_}
\frac{1}{2}-\frac{4}{d}<\frac{1}{b_{2}}<\min\left\{\frac{1}{2}+\frac{2}{d},~\frac{1}{2}-\frac{2}{d}+\frac{2\sigma}{d}\right\}.
\end{equation}
And it is obvious that there exists $(a_{2},b_{2})\in B$ satisfying \eqref{GrindEQ__3_15_}.

If $d\le 4$, then it suffices to prove that there exist $(a_{1},b_{1})\in B$ and $\bar{q}_{1}$ satisfying
\begin{equation} \nonumber
0<\frac{1}{\bar{q}_{1}}<\min\left\{1-\frac{1}{b_{2}},
~1+\frac{\sigma}{2}-\frac{4}{d}-\frac{1}{b_{2}},~\frac{1}{2}\right\},
\end{equation}
which is possible provided that
\begin{equation} \label{GrindEQ__3_16_}
\frac{1}{b_{2}}<\min\left\{1,~1+\frac{\sigma}{2}-\frac{4}{d}\right\}.
\end{equation}
Using the fact that $\sigma>\frac{8-2b}{d}$, we can easily see that there exists $(a_{2},b_{2})\in B$ satisfying \eqref{GrindEQ__3_16_}. This completes the proof of \eqref{GrindEQ__3_2_}.
\end{proof}

\begin{lemma}\label{lem 3.3.}
$d\in \mathbb N$, $0<s <\min \{2+\frac{d}{2},\frac{3}{2}d\}$, $0<b<\min\{4,d,\frac{3}{2}d-s,\frac{d}{2}+2-s\}$  and $\frac{8-2b}{d}<\sigma<\sigma_{c}(s)$. If $\sigma$ is not an even integer, assume \eqref{GrindEQ__1_11_}. Then there exist $(a_{i},b_{i})\in A$ and $(p_{i_{j}},q_{i_{j}})\in B$ $(i=3,4,~j=1,2,3)$ such that
\begin{equation} \label{GrindEQ__3_17_}
\left\|\chi(x)|x|^{-b}|u|^{\sigma}u\right\|_{L^{a_{3}'}(\mathbb R,\dot{H}_{b_{3}'}^{s-\gamma_{a_{3}',b_{3}'}-4})}
\lesssim \max_{j\in\{1,2,3\}}\left\|u\right\|_{L^{p_{3_{j}}}(\mathbb R,H_{q_{3_{j}}}^{s})}^{\sigma+1},
\end{equation}
\begin{equation} \label{GrindEQ__3_18_}
\left\|(1-\chi(x))|x|^{-b}|u|^{\sigma}u\right\|_{L^{a_{4}'}(\mathbb R,\dot{H}_{b_{4}'}^{s-\gamma_{a_{4}',b_{4}'}-4})}
\lesssim \max_{j\in\{1,2,3\}}\left\|u\right\|_{L^{p_{4_{j}}}(\mathbb R,H_{q_{4_{j}}}^{s})}^{\sigma+1},
\end{equation}
where $\chi\in C_{0}^{\infty}(\mathbb R^{d})$ is given in Remark \ref{rem 3.1.}.
\end{lemma}

\begin{proof}
Let $(a_{i},b_{i})\in A$ and $(p_{i_{j}},q_{i_{j}})\in B$ ($i=3,4,~ j=1,2,3)$ satisfy the following system:
\begin{equation} \label{GrindEQ__3_19_}
\left\{\begin{array}{lc}
{s-\gamma_{a_{i}',b_{i}'}-4\ge0},~&~~~~~~~~{(1)}\\
{\gamma_{a_{i}',b_{i}'}+4\ge 0,}~&~~~~~~~~{(2)}\\
{\frac{1}{\gamma_{i}}=\frac{1}{b_{i}'}-\sigma\left(\frac{1}{q_{i_{1}}}-\frac{s_{i}}{d}\right)-\frac{1}{\beta_{i}}},~&~~~~~~~~{(3)}\\
{\frac{1}{q_{i_{1}}}-\frac{s_{i}}{d}>0},~&~~~~~~~~{(4)}\\
{0<s_{i}<s},~&~~~~~~~~{(5)}\\
{\max\left\{\frac{1}{q_{i_{2}}}-\frac{\gamma_{a_{i}',b_{i}'}+4}{d},~0\right\}< \frac{1}{\beta_{i}}<\frac{1}{q_{i_{2}}}},~&~~~~~~~~{(6)}\\
{\frac{1}{\bar{\gamma}_{i}}=\frac{1}{b_{i}'}-(\sigma+1)\left(\frac{1}{q_{i_{3}}}-\frac{t_{i}}{d}\right)},~&~~~~~~~~{(7)}\\
{\frac{1}{q_{i_{3}}}-\frac{t_{i}}{d}>0},~&~~~~~~~~{(8)}\\
{0<t_{i}<s},~&~~~~~~~~{(9)}\\
{\frac{1}{a_{i}'}=\frac{\sigma}{p_{i_{1}}}+\frac{1}{p_{i_{2}}}=\frac{\sigma+1}{p_{i_{3}}}},~&~~~~~~~~{(10)}\\
\end{array}\right.
\end{equation}
where
\begin{equation} \label{GrindEQ__3_20_}
\left\{\begin{array}{l}
{\frac{1}{\gamma_{3}}>\frac{b}{d}},~\frac{1}{\bar{\gamma}_{3}}>\frac{b+s-\gamma_{a_{i}',b_{i}'}-4}{d},\\
{0<\frac{1}{\gamma_{4}}<\frac{b}{d}},~0<\frac{1}{\bar{\gamma}_{4}}<\frac{b+s-\gamma_{a_{i}',b_{i}'}-4}{d}.\\
\end{array}\right.
\end{equation}
Using Lemma \ref{lem 2.1.}, Lemma \ref{lem 2.2.} and Corollary \ref{cor 2.3.}, it follows from the equations (2), (4), (5), (6), (8) and (9) in \eqref{GrindEQ__3_19_} that
\begin{equation} \label{GrindEQ__3_21_}
H_{q_{i_{1}} }^{s}\hookrightarrow L^{\alpha_{1}},~H_{q_{i_{2}} }^{s}\hookrightarrow\dot{H}_{\beta_{i} }^{s-\gamma_{a_{i}',b_{i}'}-4}, ~H_{q_{i_{3}} }^{s}\hookrightarrow L^{\bar{\alpha}_{i}},
\end{equation}
where
\begin{equation} \label{GrindEQ__3_22_}
\frac{1}{\alpha_{i}}:=\frac{1}{q_{i_{1}}}-\frac{s_{i}}{d}, ~
\frac{1}{\bar{\alpha}_{i}}:=\frac{1}{q_{i_{3}}}-\frac{t_{i}}{d}.
\end{equation}
Furthermore, using Remark \ref{rem 3.1.} and \eqref{GrindEQ__3_20_}, we have
\begin{equation} \label{GrindEQ__3_23_}
\chi(x)|x|^{-b} \in L^{\gamma _{3}}, ~ \chi(x)|x|^{-b}\in  \dot{H}_{\bar{\gamma}_{3} }^{s-\gamma_{a_{3}',b_{3}'}-4},
\end{equation}
\begin{equation} \label{GrindEQ__3_24_}
(1-\chi(x))|x|^{-b} \in L^{\gamma _{4}}, ~ (1-\chi(x))|x|^{-b}\in  \dot{H}_{\bar{\gamma}_{4} }^{s-\gamma_{a_{4}',b_{4}'}-4}.
\end{equation}
Putting $\frac{1}{r_{i}}:=\frac{1}{b_{i}'}-\frac{1}{\gamma_{i}},~\frac{1}{\bar{r}_{i}}:=\frac{1}{b_{i}'}-\frac{1}{\bar{\gamma}_{i}}$, and using the equations (1), (3), (7) in \eqref{GrindEQ__3_19_}, \eqref{GrindEQ__3_21_}--\eqref{GrindEQ__3_24_}, Lemma \ref{lem 2.6.} and Corollary \ref{cor 2.5.}, we have
\begin{eqnarray}\begin{split} \label{GrindEQ__3_25_}
\left\| \chi(x)|x|^{-b} |u|^{\sigma}u\right\| _{\dot{H}_{b_{3}'}^{s-\gamma_{a_{3}',b_{3}'}-4}}
&\lesssim \left\| \chi(x)|x|^{-b} \right\| _{L^{\gamma _{3}}} \left\| |u|^{\sigma}u\right\| _{\dot{H}_{r_{3} }^{s-\gamma_{a_{3}',b_{3}'}-4} }\\
&~+\left\| \chi(x)|x|^{-b} \right\| _{\dot{H}_{\bar{\gamma}_{3} }^{s-\gamma_{a_{3}',b_{3}'}-4} } \left\| |u|^{\sigma}u\right\| _{L^{\bar{r}_{3} } }\\
&\lesssim \left\| u\right\|^{\sigma}_{L^{\alpha_{3}}}\left\|u\right\| _{\dot{H}_{\beta_{3} }^{s-\gamma_{a_{3}',b_{3}'}-4} }
+\left\| u\right\|^{\sigma+1} _{L^{\bar{\alpha}_{3}}}\\
&\lesssim \left\|u\right\|^{\sigma}_{H_{q_{3_{1}} }^{s}}\left\|u\right\|_{H_{q_{3_{2}} }^{s}}+\left\|u\right\|^{\sigma+1}_{H_{q_{3_{3}} }^{s}},
\end{split}\end{eqnarray}
and
\begin{eqnarray}\begin{split} \label{GrindEQ__3_26_}
\left\| (1-\chi(x))|x|^{-b} |u|^{\sigma}u\right\| _{\dot{H}_{b_{4}'}^{s-\gamma_{a_{4}',b_{4}'}-4}}
&\lesssim \left\| (1-\chi(x))|x|^{-b} \right\| _{L^{\gamma _{4}}} \left\| |u|^{\sigma}u\right\| _{\dot{H}_{r_{4} }^{s-\gamma_{a_{4}',b_{4}'}-4} }\\
&~+\left\| (1-\chi(x))|x|^{-b} \right\| _{\dot{H}_{\bar{\gamma}_{4} }^{s-\gamma_{a_{4}',b_{4}'}-4} } \left\| |u|^{\sigma}u\right\| _{L^{\bar{r}_{4} } }\\
&\lesssim \left\| u\right\|^{\sigma}_{L^{\alpha_{4}}}\left\|u\right\| _{\dot{H}_{\beta_{4} }^{s-\gamma_{a_{4}',b_{4}'}-4} }
          +\left\| u\right\|^{\sigma+1} _{L^{\bar{\alpha}_{4}}}\\
&\lesssim \left\|u\right\|^{\sigma}_{H_{q_{4_{1}} }^{s}}\left\|u\right\|_{H_{q_{4_{2}} }^{s}}+\left\|u\right\|^{\sigma+1}_{H_{q_{4_{3}} }^{s}},
\end{split}\end{eqnarray}
where we need the assumption that either $\sigma$ is an even integer, or
\begin{equation} \label{GrindEQ__3_27_}
\sigma\ge \lceil{ s-\gamma_{a_{i}',b_{i}'}-4}\rceil-1 ~(i=3,4).
\end{equation}
Using the H\"{o}lder inequality, \eqref{GrindEQ__3_25_}, \eqref{GrindEQ__3_26_} and equation (10) in the system \eqref{GrindEQ__3_19_}, we immediately get \eqref{GrindEQ__3_1_} and \eqref{GrindEQ__3_2_}.
Therefore, it remains to choose $(a_{i},b_{i})\in A$ and $(p_{i_{j}},q_{i_{j}})\in B$ $(i=3,4,~j=1,2,3)$ satisfying \eqref{GrindEQ__3_19_} and \eqref{GrindEQ__3_20_}.
In fact, due to \eqref{GrindEQ__1_10_}, we can see that
\begin{equation} \label{GrindEQ__3_28_}
\left\{\begin{array}{l}
{s-\gamma_{a_{i}',b_{i}'}-4\ge 0 \Leftrightarrow \frac{4}{a_{i}d}\le \frac{s}{d}+\frac{1}{2}-\frac{1}{b_{i}},}\\
{\gamma_{a_{i}',b_{i}'}+4\ge 0 \Leftrightarrow \frac{4}{a_{i}d}\ge \frac{1}{2}-\frac{1}{b_{i}}.}
\end{array}\right.
\end{equation}
The equation (10) in \eqref{GrindEQ__3_19_} implies that
\begin{equation} \label{GrindEQ__3_29_}
\frac{\sigma}{q_{i_{1}}}+\frac{1}{q_{i_{2}}}=\frac{\sigma+1}{q_{i_{3}}}=\frac{\sigma+1}{2}-\frac{4}{d}+\frac{4}{a_{i}d}.
\end{equation}
In view of \eqref{GrindEQ__3_29_}, we can easily see that $(p_{i_{1}},q_{i_{1}})\in B$ satisfies (4) in \eqref{GrindEQ__3_19_}, provided that
\begin{equation} \label{GrindEQ__3_30_}
\frac{1}{2}-\frac{4}{d}+\frac{4}{a_{i}d}\le \frac{1}{q_{i_{2}}}<\min\left\{\frac{\sigma+1}{2}-\frac{4}{d}+\frac{4}{a_{i}d}-\frac{\sigma s_{i}}{d},~\frac{\sigma+1}{2}-\frac{4}{d}+\frac{4}{a_{i}d}-\frac{\sigma(d-4)}{2d}\right\}.
\end{equation}
We can also see that $(p_{i_{3}},q_{i_{3}})\in B$ satisfies (8) in \eqref{GrindEQ__3_19_}, provided that
\begin{equation} \label{GrindEQ__3_31_}
\frac{4}{a_{i}d}>\max\left\{\frac{4}{d}-\frac{\sigma+1}{2},~\frac{2-2\sigma}{d}\right\},
\end{equation}
and
\begin{equation} \label{GrindEQ__3_32_}
\frac{(\sigma+1)t_{i}}{d}<\frac{\sigma+1}{2}-\frac{4}{d}+\frac{4}{a_{i}d}.
\end{equation}

{\bf Step 1.} In order to prove \eqref{GrindEQ__3_17_}, it suffices to choose $(a_{3},b_{3})\in A$ and $(p_{3_{j}},q_{3_{j}})\in B$ $(j=1,2,3)$ satisfying \eqref{GrindEQ__3_28_}--\eqref{GrindEQ__3_32_}, the equations (3), (5)--(7), (9) in the system \eqref{GrindEQ__3_19_} and the first equation in \eqref{GrindEQ__3_20_}.
In view of \eqref{GrindEQ__3_29_}, the equations (3), (6) in \eqref{GrindEQ__3_19_} and the first equation in \eqref{GrindEQ__3_20_}, we have
\begin{equation} \nonumber
\max\left\{0,~\frac{1}{2}+\frac{1}{q_{3_2}}-\frac{1}{b_{3}}-\frac{4}{a_{3}d}\right\}<\frac{1}{\beta_{3}}<\min\left\{\frac{1}{q_{3_2}},
~\frac{1-\sigma}{2}+\frac{4-b}{d}+\frac{\sigma s_{3}}{d}+\frac{1}{q_{3_2}}-\frac{1}{b_{3}}-\frac{4}{a_{3}d}\right\},
\end{equation}
which is possible provided that
\begin{equation} \label{GrindEQ__3_33_}
\frac{1}{q_{3_2}}>\frac{\sigma-1}{2}-\frac{4-b}{d}-\frac{\sigma s_{3}}{d}+\frac{1}{b_{3}}+\frac{4}{a_{3}d},
\end{equation}
\begin{equation} \label{GrindEQ__3_34_}
\frac{\sigma s_{3}}{d}>\frac{\sigma}{2}-\frac{4-b}{d},
\end{equation}
and
\begin{equation} \label{GrindEQ__3_35_}
\frac{4}{a_{3}d}>\frac{1}{2}-\frac{1}{b_{3}}.
\end{equation}
We can easily check that there exists $(p_{3_2},q_{3_2})\in B$ satisfying \eqref{GrindEQ__3_30_} and \eqref{GrindEQ__3_33_}, provided that
\begin{equation} \label{GrindEQ__3_36_}
\frac{\sigma s_{3}}{d}>\max\left\{\frac{\sigma}{2}-\frac{d+4-b}{d}+\frac{1}{b_{3}}+\frac{4}{a_{3}d},~\frac{\sigma(d-4)}{2d}-1+\frac{1}{b_{3}}+\frac{b}{d}\right\},
\end{equation}
\begin{equation} \label{GrindEQ__3_37_}
\frac{\sigma s_{3}}{d}<\min\left\{\frac{\sigma}{2}-\frac{2}{d}+\frac{4}{a_{3}d},~\frac{\sigma+1}{2}-\frac{4}{d}+\frac{4}{a_{3}d}\right\},
\end{equation}
\begin{equation} \label{GrindEQ__3_38_}
\frac{4}{a_{3}d}>\max\left\{-\frac{1}{2}-\frac{2\sigma}{d}+\frac{4}{d},~\frac{2-2\sigma}{d}\right\},
\end{equation}
and
\begin{equation} \label{GrindEQ__3_39_}
\frac{1}{b_{3}}<1-\frac{b}{d}.
\end{equation}
Using the fact $\frac{8-2b}{d}<\sigma<\sigma_{c}(s)$, we can easily verify that there exists $s_{3}$ satisfying \eqref{GrindEQ__3_34_}, \eqref{GrindEQ__3_36_} and \eqref{GrindEQ__3_37_} and the equation (5) in \eqref{GrindEQ__3_19_}, provided that $(a_{3},b_{3})$ satisfies \eqref{GrindEQ__3_35_}, \eqref{GrindEQ__3_38_}, \eqref{GrindEQ__3_39_} and
\begin{equation} \label{GrindEQ__3_40_}
\frac{b-2}{d}<\frac{4}{a_{3}d}<1-\frac{1}{b_{3}}+\frac{8-2b-\sigma(d-2s)}{2d},
\end{equation}
\begin{equation} \label{GrindEQ__3_41_}
\frac{1}{b_{3}}<\frac{\sigma(2s+4-d)}{2d}+\frac{d-b}{d}.
\end{equation}
On the other hand, in view of (7) in the system \eqref{GrindEQ__3_19_}, \eqref{GrindEQ__3_20_} and \eqref{GrindEQ__3_29_}, we have
\begin{equation} \label{GrindEQ__3_42_}
\frac{(\sigma+1)t_{3}}{d}>\frac{\sigma}{2}+\frac{b+s-4}{d}.
\end{equation}
One can easily check that there exist $t_{3}$ satisfying the equation (9) in the system \eqref{GrindEQ__3_19_}, \eqref{GrindEQ__3_32_} and \eqref{GrindEQ__3_42_}, provided that $\frac{8-2b}{d}<\sigma<\sigma_{c}(s)$ and
\begin{equation} \label{GrindEQ__3_43_}
\frac{4}{a_{3}d}>\frac{b+s}{d}-\frac{1}{2}.
\end{equation}
In view of \eqref{GrindEQ__1_7_}, \eqref{GrindEQ__3_28_}, \eqref{GrindEQ__3_31_}, \eqref{GrindEQ__3_35_}, \eqref{GrindEQ__3_38_}--\eqref{GrindEQ__3_41_} and \eqref{GrindEQ__3_43_}, it suffices to choose $(a_{3},b_{3})$ satisfying
\begin{equation} \label{GrindEQ__3_44_}
\left\{\begin{array}{l}
{\max\left\{\frac{2-2\sigma}{d},~\frac{1}{2}-\frac{1}{b_{3}},~-\frac{1}{2}-\frac{2\sigma}{d}+\frac{4}{d},~\frac{b-2}{d},~\frac{b+s}{d}-\frac{1}{2}\right\}
<\frac{4}{a_{3}d}}<1-\frac{1}{b_{3}}+\frac{8-2b-\sigma(d-2s)}{2d},\\
{\frac{4}{a_{3}d}\le \min\left\{\frac{s}{d}+\frac{1}{2}-\frac{1}{b_{3}},~\frac{2}{d},~1-\frac{2}{b_{3}}\right\},}\\
{0<\frac{1}{b_{3}}<\min\left\{\frac{1}{2},~1-\frac{b}{d},~\frac{\sigma(2s+4-d)}{2d}+\frac{d-b}{d}\right\}.}\\
\end{array}\right.
\end{equation}

First, we consider the case $s\ge 1$. We will choose $(a_{3},b_{3})\in S$ satisfying \eqref{GrindEQ__3_44_}.
Since $\frac{4}{a_{3}d}=1-\frac{2}{b_{3}}$, $s\ge 1$ and $\sigma<\sigma_{c}(s)$, it suffices to choose $b_{3}$ satisfying the third equation in \eqref{GrindEQ__3_44_} and
\begin{equation} \label{GrindEQ__3_45_}
\frac{1}{2}-\frac{1}{d}\le \frac{1}{b_{3}}<\min\left\{\frac{3}{4}+\frac{\sigma}{d}-\frac{2}{d},~\frac{1}{2}-\frac{1-\sigma}{d},~\frac{1}{2}-\frac{b-2}{2d},~\frac{3}{4}-\frac{b+s}{2d}
\right\}.
\end{equation}
If $d\ge 3$, then we can easily see that $b_{3}:=\frac{2d}{d-2}$ satisfies the third equation in \eqref{GrindEQ__3_44_} and \eqref{GrindEQ__3_45_}, due to the fact that $\sigma>\frac{8-2b}{d}$ and $b<\min\left\{4,~\frac{d}{2}+2-s\right\}$. And the assumption \eqref{GrindEQ__3_27_} becomes
\begin{equation} \label{GrindEQ__3_46_}
\sigma\ge \lceil{s}\rceil-2.
\end{equation}
Similarly, if $d=1,2$, then we can easily verify that
$\frac{1}{b_{3}}:=0^{+}$ satisfies the third equation in \eqref{GrindEQ__3_44_} and \eqref{GrindEQ__3_45_}, using the fact that $\sigma>\frac{8-2b}{d}$ and $b<\frac{3}{2}d-s$.
And the assumption \eqref{GrindEQ__3_27_} becomes
\begin{equation} \label{GrindEQ__3_47_}
\sigma\ge \left[s-\frac{d}{2}\right].
\end{equation}

Next, we consider the case $0<s<1$. Using the fact that $b<\min\left\{4,~\frac{d}{2}+2-s,~\frac{3}{2}d-s\right\}$ and $\frac{8-2b}{d}<\sigma<\sigma_{c}(s)$, we can easily prove the existence of $(a_{3},b_{3})$ satisfying \eqref{GrindEQ__3_44_}, whose proof will be omitted.
And the assumption \eqref{GrindEQ__3_27_} becomes trivial, since $s<1$.

{\bf Step 2.} In order to prove \eqref{GrindEQ__3_19_}, it suffices to choose $(a_{4},b_{4})\in A$ and $(p_{4_{j}},q_{4_{j}})\in B$ $(j=1,2,3)$ satisfying \eqref{GrindEQ__3_28_}--\eqref{GrindEQ__3_32_}, the equations (3), (5)--(7), (9) in the system \eqref{GrindEQ__3_19_} and the second equation in \eqref{GrindEQ__3_20_}.
We can deduce from \eqref{GrindEQ__3_29_}, the equations (3), (6) in \eqref{GrindEQ__3_19_} and the second equation in \eqref{GrindEQ__3_20_} that
\begin{equation} \nonumber
\left\{\begin{array}{l}
{\max\left\{0,~\frac{1}{2}+\frac{1}{q_{4_2}}-\frac{1}{b_{4}}-\frac{4}{a_{4}d},~\frac{1-\sigma}{2}+\frac{4-b}{d}+\frac{\sigma s_{4}}{d}+\frac{1}{q_{4_2}}-\frac{1}{b_{4}}-\frac{4}{a_{4}d}\right\}<\frac{1}{\beta_{4}}},\\
{\frac{1}{\beta_{4}}<\min\left\{\frac{1}{q_{4_2}},
~\frac{1-\sigma}{2}+\frac{4}{d}+\frac{\sigma s_{4}}{d}+\frac{1}{q_{4_2}}-\frac{1}{b_{4}}-\frac{4}{a_{4}d}\right\},}
\end{array}\right.
\end{equation}
which is possible provided that
\begin{equation} \label{GrindEQ__3_48_}
\frac{1}{q_{4_2}}>\frac{\sigma-1}{2}-\frac{4}{d}-\frac{\sigma s_{4}}{d}+\frac{1}{b_{4}}+\frac{4}{a_{4}d},
\end{equation}
\begin{equation} \label{GrindEQ__3_49_}
\frac{\sigma}{2}-\frac{4}{d}<\frac{\sigma s_{4}}{d}<\frac{\sigma-1}{2}-\frac{4-b}{d}+\frac{1}{b_{4}}+\frac{4}{a_{4}d},
\end{equation}
and
\begin{equation} \label{GrindEQ__3_50_}
\frac{4}{a_{4}d}>\frac{1}{2}-\frac{1}{b_{4}}.
\end{equation}
We can easily check that there exists $(p_{4_2},q_{4_2})\in B$ satisfying \eqref{GrindEQ__3_30_} and \eqref{GrindEQ__3_48_}, provided that
\begin{equation} \label{GrindEQ__3_51_}
\frac{\sigma s_{4}}{d}>\max\left\{\frac{\sigma}{2}-\frac{d+4}{d}+\frac{1}{b_{4}}+\frac{4}{a_{4}d},~\frac{\sigma(d-4)}{2d}-1+\frac{1}{b_{4}}\right\},
\end{equation}
\begin{equation} \label{GrindEQ__3_52_}
\frac{\sigma s_{4}}{d}<\min\left\{\frac{\sigma}{2}-\frac{2}{d}+\frac{4}{a_{4}d},~\frac{\sigma+1}{2}-\frac{4}{d}+\frac{4}{a_{4}d}\right\},
\end{equation}
and
\begin{equation} \label{GrindEQ__3_53_}
\frac{4}{a_{4}d}>\max\left\{-\frac{1}{2}-\frac{2\sigma}{d}+\frac{4}{d},~\frac{2-2\sigma}{d}\right\}.
\end{equation}
Using the fact $\frac{8-2b}{d}<\sigma<\sigma_{c}(s)$, we can easily verify that there exists $s_{4}$ satisfying \eqref{GrindEQ__3_49_}, \eqref{GrindEQ__3_51_} and \eqref{GrindEQ__3_52_} and the equation (5) in \eqref{GrindEQ__3_19_}, provided that $(a_{4},b_{4})$ satisfies \eqref{GrindEQ__3_31_}, \eqref{GrindEQ__3_50_}, \eqref{GrindEQ__3_53_} and
\begin{equation} \label{GrindEQ__3_54_}
\frac{2}{d}-\frac{\sigma}{2}<\frac{4}{a_{4}d}<1+\frac{4}{d}-\frac{1}{b_{4}}-\frac{\sigma(d-2s)}{2d},
\end{equation}
\begin{equation} \label{GrindEQ__3_55_}
\frac{1}{b_{4}}<1+\frac{\sigma(2s+4-d)}{2d}.
\end{equation}
On the other hand, in view of (7) in the system \eqref{GrindEQ__3_19_}, \eqref{GrindEQ__3_20_} and \eqref{GrindEQ__3_29_}, we have
\begin{equation} \label{GrindEQ__3_56_}
\frac{\sigma-1}{2}-\frac{4}{d}+\frac{4}{a_{4}d}+\frac{1}{b_{4}}<\frac{(\sigma+1)t_{4}}{d}<\frac{\sigma}{2}+\frac{b+s-4}{d}.
\end{equation}
We can easily see that there exist $t_{4}$ satisfying the equation (9) in the system \eqref{GrindEQ__3_19_}, \eqref{GrindEQ__3_32_} and \eqref{GrindEQ__3_56_}, provided that $\frac{8-2b}{d}<\sigma<\sigma_{c}(s)$ and $(a_{4}, b_{4})$ satisfies \eqref{GrindEQ__3_31_} and
\begin{equation} \label{GrindEQ__3_57_}
\frac{4}{a_{4}d}<\min\left\{\frac{(\sigma+1)s}{d}+\frac{4}{d}-\frac{\sigma-1}{2}-\frac{1}{b_{4}},~\frac{b+s}{d}+\frac{1}{2}-\frac{1}{b_{4}}\right\}.
\end{equation}
In view of \eqref{GrindEQ__1_7_}, \eqref{GrindEQ__3_28_}, \eqref{GrindEQ__3_31_}, \eqref{GrindEQ__3_50_}, \eqref{GrindEQ__3_53_}--\eqref{GrindEQ__3_55_} and \eqref{GrindEQ__3_57_}, it suffices to choose $(a_{4},b_{4})$ satisfying
\begin{equation} \label{GrindEQ__3_58_}
\left\{\begin{array}{l}
{\frac{4}{a_{4}d}}>\max\left\{\frac{2-2\sigma}{d},~\frac{1}{2}-\frac{1}{b_{4}},~-\frac{1}{2}-\frac{2\sigma}{d}+\frac{4}{d},~\frac{4}{d}-\frac{\sigma+1}{2},
~\frac{2}{d}-\frac{\sigma}{2}\right\},\\
{\frac{4}{a_{4}d}<\min\left\{1+\frac{4}{d}-\frac{1}{b_{4}}-\frac{\sigma(d-2s)}{2d},
~\frac{1}{2}+\frac{4}{d}-\frac{1}{b_{4}}-\frac{\sigma}{2}+\frac{(\sigma+1)s}{d}\right\}},\\
{\frac{4}{a_{4}d}\le \min\left\{\frac{s}{d}+\frac{1}{2}-\frac{1}{b_{4}},~\frac{2}{d},~1-\frac{2}{b_{4}}\right\},}\\
{0<\frac{1}{b_{4}}<\min\left\{\frac{1}{2},~1+\frac{\sigma(2s+4-d)}{2d}\right\}.}\\
\end{array}\right.
\end{equation}
We divide the study in two cases: $s\ge 1$ and $0<s<1$.

First, we consider the case $s\ge 1$. We will choose $(a_{4},b_{4})\in S$ satisfying \eqref{GrindEQ__3_58_}.
Since $\frac{4}{a_{4}d}=1-\frac{2}{b_{4}}$, $s\ge 1$ and $\sigma<\sigma_{c}(s)$, it suffices to choose $b_{4}$ satisfying the fourth equation in \eqref{GrindEQ__3_58_} and
\begin{equation} \label{GrindEQ__3_59_}
\left\{\begin{array}{l}
{\frac{1}{2}-\frac{1}{d}\le \frac{1}{b_{4}}<\min\left\{\frac{1}{2}-\frac{1}{d}+\frac{1}{\max\{4,d\}},~\frac{3}{4}-\frac{2}{d}+\frac{\sigma}{\max\{4,d\}}\right\},}\\
{\frac{1}{b_{4}}>\frac{(d-2s)\sigma}{2d}-\frac{4}{d}+\frac{1}{2}-\frac{s}{d}.}\\
\end{array}\right.
\end{equation}
If $d\ge 3$, then we can easily see that $b_{4}:=\frac{2d}{d-2}$ satisfies the fourth equation in \eqref{GrindEQ__3_58_} and \eqref{GrindEQ__3_59_}, due to the fact that $\frac{8-2b}{d}<\sigma<\sigma_{c}(s)$ and $b<\min\left\{4,~\frac{d}{2}+2-s\right\}$. And the assumption \eqref{GrindEQ__3_27_} becomes \eqref{GrindEQ__3_46_}.
Similarly, if $d=1,2$, then we can easily verify that $\frac{1}{b_{4}}:=(0)^{+}$ satisfies the fourth equation in \eqref{GrindEQ__3_58_} and \eqref{GrindEQ__3_59_}, using the fact that $\frac{8-2b}{d}<\sigma<\sigma_{c}(s)$ and $b<\frac{3}{2}d-s$. And the assumption \eqref{GrindEQ__3_27_} becomes
\eqref{GrindEQ__3_47_}.

Next, we consider the case $0<s<1$. Using the fact that $b<\min\left\{4,~\frac{d}{2}+2-s,~\frac{3}{2}d-s\right\}$ and $\frac{8-2b}{d}<\sigma<\sigma_{c}(s)$, we can easily prove the existence of $(a_{4},b_{4})$ satisfying \eqref{GrindEQ__3_58_}, whose proof will be omitted.
And the assumption \eqref{GrindEQ__3_27_} becomes trivial, due to the fact $s<1$.
\end{proof}

Using Lemma \ref{lem 2.7.} (Strichartz estimates) and the above nonlinear estimates (Lemmas \ref{lem 3.2.} and \ref{lem 3.3.}), we prove Theorem \ref{thm 1.1.}.
\begin{proof}[{\bf Proof of Theorem \ref{thm 1.1.}}]
Let $M>0$ which will be chosen later. We define the complete metric space $(D,d)$ as follows:
\begin{equation} \nonumber
D=\left\{u\in L^{p_{1}}(\mathbb R, H_{q_{1}}^{s}):~\left\|u\right\|_{X} \le M\right\},
\end{equation}
\begin{equation} \nonumber
\left\|u\right\|_{X}:=\sum_{i=1}^{2}{\left(\left\|u\right\|_{L^{p_{i}}(\mathbb R, H_{q_{i}}^{s})}+\left\|u\right\|_{L^{\bar{p}_{i}}(\mathbb R, H_{\bar{q}_{i}}^{s})}\right)}+\sum_{k=3}^{4}{\sum_{j=1}^{3}{\left\|u\right\|_{L^{p_{k_{j}}}(\mathbb R,H_{q_{k_{j}}}^{s})}}},
\end{equation}
\begin{equation} \nonumber
d(u,v):=\sum_{i=1}^{2}{\left(\left\|u\right\|_{L^{p_{i}}(\mathbb R, L^{q_{i}})}+\left\|u\right\|_{L^{\bar{p}_{i}}(\mathbb R, L^{\bar{q}_{i}})}\right)}+\sum_{k=3}^{4}{\sum_{j=1}^{3}{\left\|u\right\|_{L^{p_{k_{j}}}(\mathbb R,L^{q_{k_{j}}})}}},
\end{equation}
where $(p_{i},q_{i})\in B$, $(\bar{p}_{i},\bar{q}_{i})\in B$ $(i=1,2)$ are given in Lemma \ref{lem 3.2.} and $(p_{k_{j}},q_{k_{j}})\in B$ $(k=3,4,~j=1,2,3)$ are given in Lemma \ref{lem 3.3.}.
Now we consider the mapping
\begin{equation}\label{GrindEQ__3_60_}
G:~u(t)\to e^{it\Delta^{2}}u_{0} -i\lambda \int _{0}^{t}e^{i(t-\tau)\Delta^{2}}|x|^{-b}|u(\tau)|^{\sigma}u(\tau)d\tau .
\end{equation}
Lemma \ref{lem 2.7.} (Strichartz estimates) yields that
\begin{eqnarray}\begin{split} \label{GrindEQ__3_61_}
\left\|Gu\right\|_{X}
&\lesssim \left\| u_{0} \right\| _{H^{s} }+\left\|\chi(x)|x|^{-b}|u|^{\sigma}u\right\|_{L^{a_{1}'}(\mathbb R, L^{b_{1}'})}\\
&~+\left\|(1-\chi(x))|x|^{-b}|u|^{\sigma}u\right\|_{L^{a_{2}'}(\mathbb R, L^{b_{2}'})}\\
&~+\left\|\chi(x)|x|^{-b}|u|^{\sigma}u\right\|_{L^{a_{3}'}(\mathbb R,\dot{H}_{b_{3}'}^{s-\gamma_{a_{3}',b_{3}'}-4})}\\
&~+\left\|(1-\chi(x))|x|^{-b}|u|^{\sigma}u\right\|_{L^{a_{4}'}(\mathbb R,\dot{H}_{b_{4}'}^{s-\gamma_{a_{4}',b_{4}'}-4})},
\end{split}\end{eqnarray}
where $(a_{i}, b_{i})$ $(i=\overline{1,4})$ are given in Lemmas \ref{lem 3.2.} and \ref{lem 3.3.}.
Using \eqref{GrindEQ__3_61_}, Lemmas \ref{lem 3.2.} and \ref{lem 3.3.}, we have
\begin{equation}\label{GrindEQ__3_62_}
\left\|Gu\right\|_{X} \le C \left\| u_{0} \right\| _{H^{s} }+CM^{\sigma+1}.
\end{equation}
Similarly, using Lemma \ref{lem 2.7.} (Strichartz estimates) and Lemma \ref{lem 3.3.}, we have
\begin{eqnarray}\begin{split}\label{GrindEQ__3_63_}
d(Gu,Gv)&\lesssim \left\|\chi(x)|x|^{-b}(|u|^{\sigma }+|v|^{\sigma})|u-v|\right\|_{L^{a_{1}'}(\mathbb R, L^{b_{1}'})}\\
&~+\left\|(1-\chi(x))|x|^{-b}(|u|^{\sigma }+|v|^{\sigma})|u-v|\right\|_{L^{a_{2}'}(\mathbb R, L^{b_{2}'})}\\
&\le 2CM^{\sigma}d(u,v).
\end{split}\end{eqnarray}
Put $M=2C\left\| u_{0} \right\| _{H^{s} } $ and $\delta =2\left(4C\right)^{-\frac{\sigma +1}{\sigma }}$.
If $\left\| u_{0} \right\| _{H^{s} } \le \delta $, then we have $CM^{\sigma } \le \frac{1}{4} $.
Hence it follows from \eqref{GrindEQ__3_62_} and \eqref{GrindEQ__3_63_} that $G:(D,d)\to (D,d)$ is a contraction mapping.
So there is a unique global solution $u(t,x)$ of \eqref{GrindEQ__1_1_} in $(D,d)$.
Furthermore, by Lemma \ref{lem 2.7.} (Strichartz estimates), we have \eqref{GrindEQ__1_12_}.
The proof of scattering result is standard (see e.g. \cite{AK212}) and we omitted the details.
\end{proof}


\end{document}